\documentclass[10pt,reqno]{amsart}

\usepackage{amsthm}
\usepackage{amsmath}
\usepackage{amssymb}
\usepackage{oldgerm}
\usepackage{graphicx}
\usepackage{mathrsfs}
\usepackage{enumerate}
\usepackage{MnSymbol}
\usepackage{tikz}
\usepackage{mathtools}
\usepackage{tikz-cd}

\theoremstyle{plain}
\newtheorem{thm}{Theorem}[section]
\newtheorem{lem}[thm]{Lemma}

\newtheorem{cor}[thm]{Corollary}
\theoremstyle{definition}
\newtheorem{defn}[thm]{Definition}

\newtheorem{rem}[thm]{Remark}

\newcommand*{\sheafhom}{\mathscr{H}\kern -.5pt om}

\newcommand{\pic}{\textnormal{Pic}}

\newcommand{\spann}{\textnormal{span}}

\newcommand{\defeq}{\vcentcolon=}
\newcommand{\sym}{\textnormal{Sym}}

\newcommand{\ord}{\textnormal{ord}}

\let\ker\relax
\newcommand{\ker}{\textnormal{Ker}}

\tikzset{node distance=1.2cm,auto}
\usetikzlibrary{matrix}

\begin{document}
\title[\parbox{11.5cm}{\centering{Maximal rank divisors on $\overline{\mathcal{M}}_{g,n}$}}]{Maximal rank divisors on $\overline{\mathcal{M}}_{g,n}$}
\author{\.Irfan Kadik\"oyl\"u}
\address{Humboldt-Universit\"at zu Berlin, Institut f\"ur Mathematik, 10099
Berlin}
\email{irfankadikoylu@gmail.com}
\maketitle
\begin{abstract}
We compute the class of the effective divisors on $\overline{\mathcal{M}}_{g,n}$, which are set theoretically equal to the locus of moduli points $[C,p_1,\dots ,p_n]$ where $C$ lies on a quadric under the map given by the linear series $|K_C-p_1-\dots -p_n|$. Using this divisor class we show that the moduli spaces $\overline{\mathcal{M}}_{16,8}$ and $\overline{\mathcal{M}}_{17,8}$ are of general type. We also note that the divisor classes computed in the papers \cite{FV1} and \cite{FV2} can be used to show that $\overline{\mathcal{M}}_{12,10}$ is of general type.
\end{abstract}

\section{Introduction}
Given a curve $C$ and a line bundle $L$ on it, one can consider the natural multiplication maps
\begin{equation*}
\mu_k: \sym^k H^0(C,L)\to H^0(C,L^{\otimes k}). 
\end{equation*}
Harris conjectured that these maps are of maximal rank (i.e. either injective or surjective) for general $C$ and $L$ in the range where the Brill-Noether number is non-negative \cite{Ha}. Although the general conjecture is still open, there are numerous partial results covering various cases of the conjecture. The cases where the dimension $r\defeq h^0(L)-1$ of the projective space is equal to $3,4$ or $5$, as well as the case of non-special curves (i.e. $r=\deg(L)-g$) has been verified by Ballico and Ellia (see \cite{BE} and references therein). Voisin proved the case where 
\begin{equation*}
\mathcal{O}_C(1)=K_C-A,
\end{equation*}
and $A$ is a pencil and using this, she was able to deduce the surjectivity of the Gaussian-Wahl map for generic curves \cite{V}. Farkas confirmed the conjecture when the Brill-Noether number is zero and the map $\mu_2$ is expected to be an isomorphism. By considering the locus of curves, where $\mu_2$ fails to be an isomorphism, he obtained the first infinite family of divisors in $\overline{\mathcal{M}}_g$ violating the slope conjecture \cite{Fa1}. Later, Ballico and Fontanari \cite{BF} proved the conjecture in the range, where
\begin{equation*}
\dim \sym^2 H^0(C,L)\geq \dim H^0(C,L^{\otimes 2}).
\end{equation*}
More recently, using the theory of limit linear series and degenerating to a chain of elliptic curves Liu, Osserman, Teixidor and Zhang managed to systematically treat many other cases of the conjecture (See \cite{LOTZ} for a precise statement). There is also a tropical proof of the case of quadrics by Jensen and Payne \cite{JP}.

In this paper, we use a construction analogous to \cite{Fa1} to obtain new divisor classes on $\overline{\mathcal{M}}_{g,n}$, which are singled out as the failure locus of the maximal rank conjecture. More precisely, we consider a map of vector bundles $\phi:E\to F$ over $\mathcal{M}_{g,n}$, which restricts at a moduli point $x=[C,p_1,\dots,p_n]$ to the multiplication map
\begin{equation}\label{restriction map}
\sym^2(H^0(K_C-p_1-\dots -p_n))\xrightarrow{\phi(x)} H^0(K_C^{\otimes 2}-2p_1-\dots -2p_n).
\end{equation}
We consider pairs 
\begin{equation}\label{g(t),n(t)}
\left(g(t),n(t)\right)=\left(\frac{1}{2}(t^2+5t+10),\frac{1}{2}(t^2+3t+2)\right)\textnormal{ for }t\in\mathbb{N},
\end{equation}
in which case the dimensions of both sides in (\ref{restriction map}) are equal. Since the maximal rank conjecture holds for quadrics, the locus where the vector bundle map $\phi$ fails to be an isomorphism is a divisor $\mathfrak{Quad}_{g(t),n(t)}$ in $\mathcal{M}_{g(t),n(t)}$. By taking its closure we obtain a divisor in $\overline{\mathcal{M}}_{g(t),n(t)}$ for every $t\in\mathbb{N}$.

The sequence of the pairs $g(t),n(t)$ has the following pattern:

\begin{table}[!htbp]
\centering
\label{tablo}
\begin{tabular}{|c|c|c|c|c|c|c|c|c|}
\hline
$t$ & 0 & 1 & 2  & 3  & 4  & 5  & 6  & ... \\ \hline
$g$ & 5 & 8 & 12 & 17 & 23 & 30 & 38 & ... \\ \hline
$n$ & 1 & 3 & 6  & 10 & 15 & 21 & 28 & ... \\ \hline
\end{tabular}
\end{table}

In section 2, we compute the class of $\overline{\mathfrak{Quad}}_{g(t),n(t)}$:

\begin{thm}\label{main}
The class of the divisor $\overline{\mathfrak{Quad}}_{g(t),n(t)}$ is given by the following formula:
\begin{equation*}
\overline{\mathfrak{Quad}}_{g(t),n(t)}=(8-t)\cdot\lambda+ t\cdot\sum_{j=1}^{n(t)}\psi_j-\delta_{irr}-\sum_{i,s\geq 0}b_{i:s}(t)\cdot\sum_{|S|=s}\delta_{i:S}
\end{equation*}
where
\begin{equation*}
b_{0:s}(t)=\frac{s}{2}(st+s+t-1)\textnormal{ for }s\geq 2,
\end{equation*}
\begin{equation*}
b_{1:0}(t)=t+4,\quad b_{1:s}(t)= \frac{1}{2} (s^2t+s^2 -st + s + 6) \textnormal{ for }s\geq 1,
\end{equation*}
\begin{equation*}
\textnormal{and }b_{i:s}(t)\geq 1 \textnormal{ for }2\leq i\leq g(t)/2\textnormal{ and }0\leq s\leq n(t).
\end{equation*}
\end{thm}
In the case $t=0$ we recover a well known class: Given $[C,p]\in \mathcal{M}_{5,1}$, the linear system $|K_C-p|$ embeds $C$ to $\mathbb{P}^3$ and the existence of a quadric containing $C$ is equivalent to the existence of a $g^1_3$. Therefore, $\overline{\mathfrak{Quad}}_{5,1}$ is the pullback of the Brill-Noether divisor
\begin{equation*}
BN^1_{5,3}\defeq\{[C]\in\overline{\mathcal{M}}_5\mid W^1_3(C)\neq\emptyset\}
\end{equation*}
to $\overline{\mathcal{M}}_{5,1}$. It is well known that $BN^1_{5,3}$ has the class
\begin{equation*}
BN^1_{5,3} = 8\lambda - \delta_0 - 4\delta_1 - 6\delta_2,
\end{equation*}
see \cite{EH3}. Therefore, its pullback to $\overline{\mathcal{M}}_{5,1}$ has the class
\begin{equation*}
8\lambda - \delta_{irr} - 4\delta_{1:1} - 6\delta_{2:1} - 6\delta_{3:1} - 4\delta_{4:1},
\end{equation*}
which agrees with the formula for $\overline{\mathfrak{Quad}}_{5,1}$ in Theorem \ref{main}.

In section 3, we use the class $\overline{\mathfrak{Quad}}_{g(t),n(t)}$ to study the birational geometry of $\overline{\mathcal{M}}_{g,n}$. The moduli space $\overline{\mathcal{M}}_{g,n}$ is known to be of general type when  $g\geq 24$ (\cite{HM}, \cite{EH3}) and its birational type for various $g,n$ in the range $g\leq 23$ are established in the papers \cite{Lo} and \cite{Fa1}. Using the divisor $\overline{\mathfrak{Quad}}_{g(t),n(t)}$, we manage to treat some of the unknown cases of this problem:
\begin{thm}\label{GenTypeThms}
The moduli spaces $\overline{\mathcal{M}}_{16,8}, \overline{\mathcal{M}}_{17,8}$ and $\overline{\mathcal{M}}_{12,10}$ are of general type.
\end{thm}

This work is part of my PhD thesis. I am grateful to my advisor Gavril
Farkas and my coadvisor Angela Ortega for suggesting me this problem. My thanks also go to my colleagues in Humboldt University for many helpful discussions. I have been supported by Berlin Mathematical School during the preparation of this work.
\section{The computation of the class $\overline{\mathfrak{Quad}}_{g(t),n(t)}$}

As we already pointed out in the introduction, the divisor $\overline{\mathfrak{Quad}}_{g(t),n(t)}$ is defined as the closure of the degeneracy locus of a vector bundle map over $\mathcal{M}_{g(t),n(t)}$. To extend this degeneracy locus description to the boundary, we let
\begin{equation*}
\pi :\overline{\mathcal{M}}_{g(t),n(t)+1}\to\overline{\mathcal{M}}_{g(t),n(t)}
\end{equation*}
be the map that forgets the last marked point and $\mathscr{L}$ be the cotangent line bundle on $\overline{\mathcal{M}}_{g(t),n(t)+1}$. That is, $\mathscr{L}$ is naturally isomorphic to $\omega_C$ when restricted to the fiber $\pi^{-1}\left([C,p_1,\dots,p_{n(t)}]\right)$. We let $\phi$ denote the natural multiplication map
\begin{equation}\label{vbmap}
\sym^2\left(\pi_*\mathscr{L}(-\sum_{j=1}^{n(t)}\delta_{0:\{j,n(t)+1\}})\right)\xrightarrow{\phi} \pi_*\mathscr{L}^{\otimes 2}(-2\cdot\sum_{j=1}^{n(t)}\delta_{0:\{j,n(t)+1\}}).
\end{equation}
Over a moduli point $[C,p_1,\dots , p_{n(t)}]$ where $C$ is a smooth curve, the map $\phi$ restricts to the map (\ref{restriction map}) and therefore extends the degeneracy locus structure to the boundary. 

Note that if we have
\begin{equation*}
i<s \quad\textnormal{or}\quad g(t)-i<n(t)-s,
\end{equation*}
the evaluation map
\begin{equation*}
\pi_*\mathscr{L}\xrightarrow{ev} \pi_*\left(\mathscr{L}|_{\sum_{j=1}^{n(t)}\delta_{0:\{j,n(t)+1\}}}\right)
\end{equation*}
fails to be surjective over $\Delta_{i:S}$ (Here and in what follows we set $s\defeq |S|$). We will deal with such boundary components later, for now we restrict our attention to the partial compactification $\widetilde{\mathcal{M}}_{g(t),n(t)}$, which we define as the union of $\mathcal{M}_{g(t),n(t)}$ together with boundary divisors $\Delta_{i:S}$, where
\begin{equation*}
s\leq i \quad\textnormal{and}\quad n(t)-s\leq g(t)-i.
\end{equation*}
The sheaves in (\ref{vbmap}) are locally free over $\widetilde{\mathcal{M}}_{g(t),n(t)}$ away from loci of codimension at least 2. Therefore the first degeneracy locus $D_1(\phi)$ contains the divisor $\overline{\mathfrak{Quad}}_{g(t),n(t)}\cap \widetilde{\mathcal{M}}_{g(t),n(t)}$. We use Grothendieck-Riemann-Roch formula to compute its class and obtain the following result:
\begin{thm}\label{Lambda,Psi,Deltairr}
The coefficients of $\lambda,\psi_i$ and $\delta_{irr}$ in $\overline{\mathfrak{Quad}}_{g(t),n(t)}$ are $8-t, t$ and $-1$, respectively, Moreover, $b_{i:s}(t)\geq 1$ whenever $s\leq i$ and $n(t)-s\leq g(t)-i$.
\end{thm}
\begin{proof}
On $\widetilde{\mathcal{M}}_{g(t),n(t)}$ we have the exact sequence
\begin{align*}
0 &\rightarrow \pi_*\left(\mathscr{L}(-\sum_{j=1}^{n(t)}\delta_{0:\{j,n(t)+1\}})\right)\rightarrow \pi_*\mathscr{L}\xrightarrow{ev} \pi_*\left(\mathscr{L}|_{\sum_{j=1}^{n(t)}\delta_{0:\{j,n(t)+1\}}}\right)\rightarrow \\
&\rightarrow  R^1\pi_*\left( \mathscr{L}(-\sum_{j=1}^{n(t)}\delta_{0:\{j,n(t)+1\}})\right)\rightarrow  R^1\pi_*\mathscr{L}\rightarrow  0.
\end{align*}
It is easy to see that the evaluation map $ev$ is surjective in codimension $2$ in the range $s\leq i$ and $n(t)-s\leq g(t)-i$. Since $R^1\pi_*\mathscr{L}\cong \mathcal{O}$, it follows that
\begin{equation*}
R^1\pi_*\left(\mathscr{L}(-\sum_{j=1}^{n(t)}\delta_{0:\{j,n(t)+1\}})\right)
\end{equation*}
is isomorphic to $\mathcal{O}$ in codimension 2. Since the rank of $\pi_*\mathscr{L}(-\sum_{j=1}^{n(t)}\delta_{0:\{j,n(t)+1\}})$ is equal to $g(t)-n(t)=t+4$, we have that 
\begin{equation*}
c_1\left(\sym^2\left(\pi_*\mathscr{L}(-\sum_{j=1}^{n(t)}\delta_{0:\{j,n(t)+1\}})\right)\right)=(t+5)\cdot c_1\left(\pi_*\mathscr{L}(-\sum_{j=1}^{n(t)}\delta_{0:\{j,n(t)+1\}})\right).
\end{equation*}
From the exact sequence above, it follows immediately that
\begin{equation*}
c_1\left(\pi_*\mathscr{L}(-\sum_{j=1}^{n(t)}\delta_{0:\{j,n(t)+1\}})\right)=\lambda-\sum_{j=1}^{n(t)}\psi_j.
\end{equation*}
We use Grothendieck-Riemann-Roch formula to compute that
\begin{equation*}
c_1\left(\pi_*\mathscr{L}^{\otimes 2}(-2\cdot \sum_{j=1}^{n(t)}\delta_{0:\{j,n(t)+1\}})\right)=13\lambda -5\cdot \sum_{j=1}^{n(t)}\psi_j-\delta,
\end{equation*}
where $\delta$ denotes the class of the whole boundary. By Porteous formula we get that
\begin{equation*}
[D_1(\phi)]=(8-t)\cdot\lambda +t\cdot \sum_{j=1}^{n(t)}\psi_j-\delta.
\end{equation*}
The class $[D_1(\phi)]$ is equal to the sum of $\overline{\mathfrak{Quad}}_{g(t),n(t)}$ and positive multiples of the boundary components, over which the map (\ref{vbmap}) is degenerate. Thus we obtain the bound $b_{i:s}(t)\geq 1$ whenever $s\leq i$ and $n(t)-s\leq g(t)-i$. In Theorem \ref{Nondegenerateness over Delta irr}, we will prove that (\ref{vbmap}) is generically non-degenerate over $\Delta_{irr}$, which will imply that the coefficient of $\delta_{irr}$ is equal to $-1$.
\end{proof}
To obtain a bound for $b_{i:s}(t)$ in the case when $i<s$ or $g(t)-i<n(t)-s$, we modify the sheaves in (\ref{vbmap}) as follows:

We let
\begin{equation*}
\mathscr{L}'\defeq \mathscr{L}\left(-\sum_{j=1}^{n(t)}\delta_{0:\{j,n(t)+1\}} + \sum_{\substack{0\leq i\leq g(t)\\ i< s\\ |S| = s}}(i-s-1)\cdot\delta_{i:S\cup\{n(t)+1\}}\right),
\end{equation*}
and consider the natural map
\begin{equation}\label{vbmap2}
\sym^2\left(\pi_*\mathscr{L}'\right)\xrightarrow{\phi'}\pi_*\left(\mathscr{L}'^{\otimes 2}\right).
\end{equation}
Using Grauert's Theorem it can easily be confirmed that the dimension of fibers of $\pi_*\mathscr{L}'$ and $\pi_*\left(\mathscr{L}'^{\otimes 2}\right)$ stay constant over an open subset of $\overline{\mathcal{M}}_{g(t),n(t)}$, whose complement has codimension at least 2. Therefore in codimension 2, the map $\phi'$ is a map of vector bundles and is an extension of $\phi$.

To compute the class of the degeneracy locus $[D_1(\phi')]$, we will intersect it with simple test curves, whose intersection with the generators of $\pic(\overline{\mathcal{M}}_{g,n})$ we already know. To this end, we let
\begin{equation*}
[D,q',\{p_j\mid j\in S\}]\in \mathcal{M}_{i,S\cup\{q'\}}
\end{equation*}
and
\begin{equation*}
[C,\{p_j\mid j\in S^c\}]\in \mathcal{M}_{g(t)-i,S^c}
\end{equation*}
be general pointed curves and define the test curve $T_{i:S}$ as follows:
\begin{equation*}
T_{i:S}\defeq \{[C\cup_{q\sim q'} D,p_1,\dots,p_{n(t)}]\}_{q\in C},
\end{equation*}
that is, the point of attachment moves on the curve $C$. The intersection of $T_{i:S}$ with the standard divisor classes of $\overline{\mathcal{M}}_{g,n}$ can be computed using Lemma 1.4 in \cite{AC}. We note them here for readers convenience:
\begin{enumerate}[i)]
\item $T_{i:S}\cdot \psi_j = 1$ if $j\in S^c$,
\item $T_{i:S}\cdot \delta_{i:S\cup \{j\}} = 1$ if $j\in S^c$,
\item $T_{i:S}\cdot \delta_{i:S} = -(2(g(t)-i)-2+n(t)-s)$,
\end{enumerate}
and the intersection of $T_{i:S}$ with all other generators of $\pic(\overline{\mathcal{M}}_{g,n})$ equals zero.
\begin{lem}\label{Intersection Numbers}
For $0\leq i\leq g(t)$ and $i< s$, we have the following intersection numbers:
\begin{align*}
T_{i:S}\cdot c_1(\pi_*\mathscr{L}') &=-(i-s) \left((i-s-1) (g-i-1)+n-s\right),\\
T_{i:S}\cdot c_1\left(\pi_*(\mathscr{L}'^{\otimes 2})\right) &=-2 \left(i^2 (4 g+6 s+1)+i \left(-g (6 s+5)+3 n-2 s^2+5\right)\right)\\
&-2s (g (2 s+3)-2 n-3)+8 i^3.
\end{align*}
\end{lem}
\begin{proof}
The fiber of the bundle $\pi_*\mathscr{L}'$ over the point $[C\cup_{q\sim q'} D,p_1,\dots,p_{n(t)}]\in T_{i:S}$ is equal to the sections of
\begin{equation}\label{Fiber of the extension}
H^0\left(K_C+(i-s)q-\sum_{j\in S^c}p_j\right)\oplus H^0\left(K_D+(2-i+s)q'-\sum_{j\in S}p_j)\right)
\end{equation}
that are compatible at the node $q\sim q'$. To prove the lemma we need to globalize this fibral description. To this end, we define the clutching maps
\begin{align*}
\eta_{g-i}&:\overline{\mathcal{M}}_{g-i,S^c\cup\{n+1,0\}}\times \overline{\mathcal{M}}_{i,S\cup\{0\}}\to \overline{\mathcal{M}}_{g,n+1},\\
\eta_{i}&:\overline{\mathcal{M}}_{g-i,S^c\cup\{0\}}\times \overline{\mathcal{M}}_{i,S\cup\{n+1,0\}}\to \overline{\mathcal{M}}_{g,n+1},
\end{align*}
which are defined as the maps that identify the points with the labels $0$. Clearly they map onto the boundary divisors $\Delta_{i:S}$ and $\Delta_{i:S\cup\{n+1\}}$, respectively. These boundary divisors intersect at the locus where the point with the label $n+1$ hits the node and this locus is isomorphic to the image of the clutching map
\begin{equation*}
\eta_{\Sigma}:\overline{\mathcal{M}}_{g-i,S^c\cup\{0\}}\times \overline{\mathcal{M}}_{0,\{0,n+1,-1\}}\times \overline{\mathcal{M}}_{i,S\cup\{-1\}}\to \overline{\mathcal{M}}_{g,n+1},
\end{equation*}
which identifies the points with labels $0$ and $-1$, respectively. We also have maps from the domains of these 3 clutching maps to 
\begin{equation*}
\overline{\mathcal{M}}_{g-i,S^c\cup\{0\}}\times \overline{\mathcal{M}}_{i,S\cup\{0\}},
\end{equation*}
which are defined as the maps that forget the point with label $n+1$. We denote these maps by $\pi_{g-i}, \pi_{i}$ and $\pi_{\Sigma}$, respectively. In what follows (by abuse of notation) we denote by $\pi_*\mathscr{L}'$ the pullback of its restriction to $\Delta_{i:S}\subseteq \overline{\mathcal{M}}_{g,n}$ under the clutching map
\begin{equation*}
\overline{\mathcal{M}}_{g-i,S^c\cup\{0\}}\times \overline{\mathcal{M}}_{i,S\cup\{0\}} \to \Delta_{i:S}.
\end{equation*}
The bundle $\pi_*\mathscr{L}'$ sits in the following exact sequence
\begin{equation*}
0\to \pi_*\mathscr{L}'\to {\pi_{g-i}}_*\left(\eta_{g-i}^*\mathscr{L}'\right)\oplus {\pi_i}_*\left(\eta_{i}^*\mathscr{L}'\right) \to {\pi_{\Sigma}}_*\left(\eta_{\Sigma}^*\mathscr{L}'\right)\to 0.
\end{equation*}
Therefore, we have that
\begin{equation}\label{Chern Class equality}
c_1(\pi_*\mathscr{L}') = c_1\left({\pi_{g-i}}_*(\eta_{g-i}^*\mathscr{L}')\right) + c_1\left({\pi_{i}}_*(\eta_{i}^*\mathscr{L}')\right) - c_1\left({\pi_{\Sigma}}_*(\eta_{\Sigma}^*\mathscr{L}')\right).
\end{equation}
To prove the lemma we need to compute the intersection number of these Chern classes with $T_{i:S}$. These classes are elements of
\begin{equation*}
H^2(\overline{\mathcal{M}}_{g-i,S^c\cup\{0\}},\mathbb{Q})\oplus H^2(\overline{\mathcal{M}}_{i,S\cup\{0\}},\mathbb{Q}),
\end{equation*}
and classes belonging to the second direct summand clearly have $0$ intersection with the test curve $T_{i:S}$. Therefore it suffices to compute $H^2(\overline{\mathcal{M}}_{g-i,S^c\cup\{0\}},\mathbb{Q})$ part of the Chern classes appearing in the formula (\ref{Chern Class equality}) and their intersection with the test curve $T'_{i:S}\subseteq \overline{\mathcal{M}}_{g-i,S^c\cup\{0\}}$, which is defined by fixing a general element of $\overline{\mathcal{M}}_{g-i,S^c\cup\{0\}}$ and letting the point with label $0$ vary on the curve.

Using the formula
\begin{equation*}
c_1(\mathscr{L}) = \psi_{n+1} - \sum_{j=1}^n \delta_{0:\{j,n+1\}},
\end{equation*}
we first compute that
\begin{equation*}
c_1(\mathscr{L}') = \psi_{n+1}-2\sum_{j=1}^n\delta_{0:\{j,n+1\}}+(i-s-1)\cdot\delta_{i:S\cup\{n+1\}}+\sum_{j\in S^c}(i-s-2)\cdot\delta_{i:S\cup\{j,n+1\}}+\dots
\end{equation*}
(Here the ``dots" denote the classes, which have $0$ intersection with the test curve $T_{i:S}$ and hence are irrelevant to our computation.)

Using the pullback formulas in \cite{AC} and the fact that the map $\pi_{\Sigma}$ is an isomorphism, we compute that
\begin{equation*}
c_1\left({\pi_{\Sigma}}_*(\eta_{\Sigma}^*\mathscr{L}')\right) = -(i-s-1)\cdot \psi_0 +\sum_{j\in S^c} (i-s-2)\cdot\delta_{0:\{0,j\}}+\dots\in H^2(\overline{\mathcal{M}}_{g-i,S^c\cup\{0\}},\mathbb{Q}).
\end{equation*}
To compute $c_1\left({\pi_{i}}_*(\eta_{i}^*\mathscr{L}')\right)$, we observe that
\begin{equation*}
c_1(\eta_{i}^*\mathscr{L}') = -(i-s-1)\cdot \psi_0 +\sum_{j\in S^c}(i-s-2)\cdot \delta_{0:\{0,j\}}+\dots\in H^2(\overline{\mathcal{M}}_{g-i,S^c\cup\{0\}},\mathbb{Q}).
\end{equation*}
The restriction of the bundle ${\pi_{i}}_*(\eta_{i}^*\mathscr{L}')$ to $\overline{\mathcal{M}}_{g-i,S^c\cup\{0\}}$ is a trivial bundle twisted by this class. Therefore,
\begin{equation*}
c_1\left({\pi_{i}}_*(\eta_{i}^*\mathscr{L}')\right) = \textnormal{rank}({\pi_{i}}_*(\eta_{i}^*\mathscr{L}'))\cdot c_1(\eta_{i}^*\mathscr{L}').
\end{equation*}
From the fibral description (\ref{Fiber of the extension}), it is easy to see that $\textnormal{rank}({\pi_{i}}_*(\eta_{i}^*\mathscr{L}')) = 1$. Therefore, we have that
\begin{equation*}
T'_{i:S}\cdot c_1(\pi_*\mathscr{L}') = T'_{i:S}\cdot c_1\left({\pi_{g-i}}_*(\eta_{g-i}^*\mathscr{L}')\right).
\end{equation*}
To compute this last quantity, we use Grothendieck-Riemann-Roch formula. First we compute that
\begin{align*}
&c_1(\eta_{g-i}^*\mathscr{L}') = \psi_{n+1}-2\sum_{j\in S^c}\delta_{0:\{j,n+1\}}+(i-s-1)\cdot\delta_{0:\{0,n+1\}}+\\
&+\sum_{j\in S^c}(i-s-2)\cdot \delta_{0:\{0,j,n+1\}}+\dots\in H^2(\overline{\mathcal{M}}_{g-i,S^c\cup\{n+1,0\}},\mathbb{Q}).
\end{align*}
As in the proof of Theorem \ref{Lambda,Psi,Deltairr}, one can show that
\begin{equation*}
R^1{\pi_{g-i}}_*\left(\eta_{g-i}^*\mathscr{L}'\right)\cong \mathcal{O}.
\end{equation*}
Then a standard Grothendieck-Riemann-Roch computation yields that
\begin{equation*}
T'_{i:S}\cdot c_1\left({\pi_{g-i}}_*(\eta_{g-i}^*\mathscr{L}')\right) = -(i-s) \left((i-s-1) (g-i-1)+n-s\right).
\end{equation*}
The computation of $T_{i:S}\cdot c_1\left(\pi_*(\mathscr{L}'^{\otimes 2})\right)$ is done in the exact same way and we skip these details.
\end{proof}
\begin{thm}
We have that $b_{i:s}(t)\geq 1$ for $0\leq i\leq g(t)$ and $0\leq s\leq n(t)$.
\end{thm}
\begin{proof}
In Theorem \ref{Lambda,Psi,Deltairr} we have already shown that $b_{i:s}(t)\geq 1$ whenever $s\leq i$ and $n(t)-s\leq g(t)-i$. To deal with the remaining cases, we assume that $i<s$. We consider the degeneracy locus of the map (\ref{vbmap2}) and write the relation
\begin{equation*}
[D_1(\phi')]=\overline{\mathfrak{Quad}}_{g(t),n(t)}+\sum d_{i:s}(t)\cdot\delta_{i:S},
\end{equation*}
where $d_{i:s}(t)\geq 0$. By intersecting both sides of this equality with the test curve $T_{i:S}$, we obtain the relation
\begin{equation*}
T_{i:S}.[D_1(\phi')]=(2g(t)-2i-2+n(t)-s)\tilde{b}_{i:s}(t)-(n(t)-s)\tilde{b}_{i:s+1}(t)+(n(t)-s)t,
\end{equation*}
where $\tilde{b}_{i:s}(t)\defeq b_{i:s}(t)-d_{i:s}(t)$. Since, $b_{i:s}(t)\geq \tilde{b}_{i:s}(t)$ it suffices to prove that $\tilde{b}_{i:s}(t)\geq 1$. Using Lemma \ref{Intersection Numbers} we solve this equation and obtain that
\begin{equation}\label{Tilde b_{i:s}}
\tilde{b}_{i:s}(t) = \frac{1}{2} \left(i^2 (t-3)-i (2 s (t-1)+t-5)+s (s t+s+t-1)\right).
\end{equation}
It is elementary to check that this quantity is always greater than 1.
\end{proof}
The vector bundle map (\ref{vbmap2}) is degenerate over most of the boundary divisors in $\overline{\mathcal{M}}_{g(t),n(t)}$, but it is actually generically non-degenerate over $\Delta_{0:S}$. To see this, first note that the fiber of (\ref{vbmap2}) over a general element of the test curve $T_{0:S}$ has the form
\begin{equation*}
\sym^2\left(H^0(K_C-s\cdot q-\sum_{j\in S^c}p_j)\right)\xrightarrow{\phi'}H^0\left(K_C^{\otimes 2}-2s\cdot q-2\sum_{j\in S^c}p_j\right).
\end{equation*}
In Theorem \ref{Nondegenerateness over Delta 0:k} we will prove that this map is an isomorphism if the pointed curve
\begin{equation*}
[C,q,\{p_j\mid j\in S^c\}]\in \mathcal{M}_{g(t),S^c\cup\{q\}}
\end{equation*}
is general. We first state a lemma known as ``Lemme d'Horace", which we will be using in the proof of Theorem \ref{Nondegenerateness over Delta 0:k}.
\begin{lem}\label{Horace}
Let $H\subseteq \mathbb{P}^r$ be a hyperplane and $X,Y\subseteq \mathbb{P}^r$ be reduced subschemes such that $Y\subseteq H$ and no irreducible component of $X$ lies in $H$. Then for any integer $m\geq 1$, one has a short exact sequence of ideal sheaves
\begin{equation*}
0\to \mathcal{I}_{X/\mathbb{P}^r}(m-1)\to \mathcal{I}_{X\cup Y/\mathbb{P}^r}(m)\to \mathcal{I}_{(X\cup Y)\cap H/H}(m)\to 0.
\end{equation*}
\end{lem}
\begin{proof}
See \cite{Hi}.
\end{proof}
\begin{thm}\label{Nondegenerateness over Delta 0:k}
Let $C$ be a general curve of genus $g(t)$ and $p_1,\dots , p_k$ general points on $C$. Let $a_1,\dots , a_k$ be natural numbers such that $\sum_{j=1}^k a_j=n(t)$. Then the multiplication map
\begin{equation*}
\sym^2 H^0(K_C-\sum_{j=1}^k a_jp_j)\to H^0(K_C^{\otimes 2} - \sum_{j=1}^k 2a_jp_j)
\end{equation*}
is an isomorphism.
\end{thm}

\begin{proof}
Clearly, it is sufficient to prove the theorem for the special case where all points come together, i.e. it suffices to find a pointed curve $[C,p]\in \mathcal{M}_{g(t),1}$ such that
\begin{equation*}
\sym^2 H^0(K_C-n(t)p)\to H^0(K_C^{\otimes 2} - 2n(t)p)
\end{equation*}
is an isomorphism.

To prove this we use degeneration. Let $C'$ be a genus $g(t)-2$ curve and $q_1,q_2,q_3,p$ general points on it. We consider the stable pointed curve $[X,p]\in \overline{\mathcal{M}}_{g(t),1}$, which we obtain by gluing $C'$ with a rational curve $R'$ at the points $q_1,q_2,q_3$. We have the short exact sequence
\begin{equation*}
0\to \omega_X(-n(t)p)\to \omega_{\tilde{X}}(-n(t)p)\to \mathbb{C}_{q_1}\oplus \mathbb{C}_{q_2}\oplus \mathbb{C}_{q_3}\to 0,
\end{equation*}
where $\tilde{X}$ is the normalization of $X$ and the right most map is the difference of the residues of the differentials on $C'$ and $R'$. Therefore, the space of sections $H^0(\omega_X(-n(t)p))$ is equal to the kernel of the map
\begin{equation*}
H^0(\omega_{C'}(q_1+q_2+q_3-n(t)p))\oplus H^0(\omega_{R'}(q_1+q_2+q_3))\xrightarrow{\varphi} \mathbb{C}_{q_1}\oplus \mathbb{C}_{q_2}\oplus \mathbb{C}_{q_3}.
\end{equation*}
Moreover, $\ker (\varphi)$ can be identified with $H^0(\omega_{C'}(q_1+q_2+q_3-n(t)p))$, since for any section of $H^0(\omega_{C'}(q_1+q_2+q_3-n(t)p))$ with residues $\lambda_1,\lambda_2,\lambda_3$ at $q_1,q_2,q_3$, one can find and element of $H^0(\omega_{R'}(q_1+q_2+q_3))$ having residues $-\lambda_1,-\lambda_2,-\lambda_3$ at these points, so that these sections glue to give a section of $H^0(\omega_X(-n(t)p))$. 

Therefore, the invertible sheaf $\omega_X(-n(t)p)$ is base point free and gives a map to the projective space, whose image consists of the image of $C'\to \mathbb{P}^r$ under the linear system 
\begin{equation*}
|\omega_{C'}(q_1+q_2+q_3-n(t)p)| 
\end{equation*}
(that is, $r=\dim|\omega_{C'}(q_1+q_2+q_3-n(t)p)|$) and the 3-secant line $\overline{q_1,q_2,q_3}$ (embedded by the linear system $|\omega_{R'}(q_1+q_2+q_3)|$).

Since a quadric in $\mathbb{P}^r$ containing $C'$ automatically contains the 3-secant line, to prove the theorem it is sufficient to prove that the map
\begin{equation*}
\sym^2 H^0(K_{C'}+q_1+q_2+q_3-n(t)p)\to H^0(K_{C'}^{\otimes 2}+2q_1+2q_2+2q_3 - 2n(t)p)
\end{equation*}
is an isomorphism.

To prove this we degenerate further and consider the following stable curve:
We let $R''$ be a rational curve with $r+2$ marked points on it, which are labeled as $q_1,q_2, s_1,\dots , s_r$. Let $C''$ be a curve of genus $g(t)-r-1$ with marked points $q_3,p,s_1,\dots ,s_r$. We let $[X,q_1,q_2,q_3,p]\in \overline{\mathcal{M}}_{g(t)-2,4}$ be the stable curve, which we obtain by gluing $C''$ with $R''$ at the points labeled with $s_j$. Along the same lines of reasoning as above, we observe that the linear system $\omega_X(q_1+q_2+q_3-n(t)p)$ is base point free and its image in $\mathbb{P}^r$ can be described as follows: 

The image of $R''$ is a rational normal curve in $\mathbb{P}^r$ embedded via
\begin{equation*}
|\omega_{R''}(q_1+q_2+s_1+\dots + s_r)|
\end{equation*}
and $C''$ is embedded to the hyperplane $H\defeq \spann\{s_1,\dots ,  s_r\}$ via the linear series 
\begin{equation*}
|\omega_{C''}(q_3+s_1+\dots +s_r-n(t)p)|.
\end{equation*}
Since $C''$ lies in the hyperplane and $C''\cap R''=\{s_1,\dots ,s_r\}$, by Lemma \ref{Horace},
\begin{equation*}
H^1(\mathcal{I}_{X/\mathbb{P}^r}(2))=H^1(\mathcal{I}_{{C''}/H}(2)).
\end{equation*} 
That is, the original problem is now reduced to finding a general pointed curve $[C,p,q_1,\dots ,q_{r+1}]$ of genus $g(t)-r-1$ such that
\begin{equation*}
\sym^2 H^0(K_C+\sum_{j=1}^{r+1}q_j-n(t)p)\to H^0(K_C^{\otimes 2}+2\sum_{j=1}^{r+1}q_j - 2n(t)p)
\end{equation*}
is an isomorphism. 

Note that as opposed to the first degeneration, the latter one reduces the dimension of the projective space in consideration. Using this degeneration successively (that is, in the next step we consider a pointed curve $[C''',p,q_3,\dots,q_{r+1},s_1,\dots,s_{r-1}]$ of genus $g(t)-2r+1$ glued to a pointed rational curve $[R''',q_1,q_2,s_1,\dots,s_{r-1}]$ at the points with label $s_j$), we can reduce the question to a question in $\mathbb{P}^3$. Precisely, to prove the theorem it suffices to show that the map
\begin{equation*}
\sym^2 H^0(K_C+\sum_{j=1}^{n(t)+2}q_j-n(t)p)\to H^0(K_C^{\otimes 2}+2\sum_{j=1}^{n(t)+2}q_j - 2n(t)p)
\end{equation*}
is an isomorphism for a general pointed curve $[C,q_1,\dots ,q_{n(t)+2},p]$, where the genus of $C$ is 3. (That the number of the points $q_j$ is $n(t)+2$ and the genus is 3 can be computed using the formulas in (\ref{g(t),n(t)}) and the fact that we need precisely $t$ such degenerations, since $r=g(t)-n(t)-1=t+3$).

To prove this final statement we can specialize to the case where $q_j=p$ for $j=4,\dots , n(t)+2$ and show that 
\begin{equation}\label{Reduction to genus 3}
\sym^2 H^0(K_C+q_1+q_2+q_3-p)\to H^0(K_{C}^{\otimes 2}+2q_1+2q_2+2q_3 - 2p)
\end{equation}
is an isomorphism for a general pointed genus 3 curve $[C,q_1,q_2,q_3,p]$. (Note that the union of the image of $C$ via $|K_C+q_1+q_2+q_3-p|$ with the 3 secant line $\overline{q_1,q_2,q_3}$ is an element of $\overline{\mathcal{M}}_{5,1}$ which is the $t=0$ case of our problem).

This statement (which can be confirmed also directly) is true by \cite{GL}, since $\deg(K_C+q_1+q_2+q_3-p)=6$ and it is equal to $2g(C)+1-\textnormal{Cliff(C)}$ if $C$ is not hyperelliptic. Hence, (\ref{Reduction to genus 3}) is an isomorphism, if $C$ is not hyperelliptic.
\end{proof}
\begin{cor}
We have that $b_{0:s}(t)=\frac{s}{2}(st+s+t-1)$ for $s\geq 2$.
\end{cor}
\begin{proof}
Setting $i=0$ in (\ref{Tilde b_{i:s}}) we obtain the claimed formula.
\end{proof}
Using the same methods as in the proof of Theorem \ref{Nondegenerateness over Delta 0:k}, we prove the following theorem, which finishes the computation of the coefficient of $\delta_{irr}$: 

\begin{thm}\label{Nondegenerateness over Delta irr}
The vector bundle map (\ref{vbmap}) is generically non-degenerate over $\Delta_{irr}$.
\end{thm}
\begin{proof}
To prove the theorem it is sufficient to exhibit a curve $C$ of genus $g(t)-1$ with marked points $q_1,q_2,p_1,\dots, p_{n(t)}$ such that the image of $C$ under the map given by $|K_C+q_1+q_2-p_1-\dots -p_{n(t)}|$ does not lie on any quadric. To prove this we will use the same type of degenerations that we used in the proof of Theorem \ref{Nondegenerateness over Delta 0:k}. We let $[X,q_1,q_2,p_1,\dots,p_{n(t)}]\in \overline{\mathcal{M}}_{g(t)-1,n(t)+2}$ be the stable curve which we obtain by gluing a pointed rational curve $[R,q_1,q_2,s_1,\dots,s_r]$ with a pointed genus $g(t)-r$ curve $[C,s_1,\dots,s_r,p_1,\dots,p_{n(t)}]$ at the points with the same labels $s_j$ (As before $r=g(t)-n(t)-1=t+3$). The image of $X$ under $|\omega_X(q_1+q_2-p_1-\dots -p_{n(t)})|$ is again the union of the rational curve $R$ embedded to $\mathbb{P}^r$ via
\begin{equation*}
|\omega_R(q_1+q_2+s_1+\dots +s_r)|
\end{equation*}
with the curve $C$ embedded to the hyperplane $H\defeq\spann\{s_1,\dots,s_r\}$ via
\begin{equation*}
|\omega_C+s_1+\dots+s_r-p_1-\dots-p_{n(t)}|.
\end{equation*}
By Lemma \ref{Horace}, we obtain $H^1(\mathcal{I}_{X/\mathbb{P}^r}(2))=H^1(\mathcal{I}_{{C}/H}(2))$, which again reduces the problem showing that $C\subseteq H$ does not lie on any quadrics. As in the proof of Theorem \ref{Nondegenerateness over Delta 0:k}, we keep degenerating in this manner until the question is reduced to proving that for a general pointed genus 4 curve $[C,q_1,\dots,q_{n(t)+1},p_1,\dots,p_{n(t)}]$ the image of $C$ under the linear system
\begin{equation*}
|K_C+q_1+\dots+q_{n(t)+1}-p_1-\dots-p_{n(t)}| 
\end{equation*}
does not lie on any quadrics. Specializing to the case $q_j=p_j$ for $j=1,\dots,n(t)-1$ reduces our problem to finding a pointed genus 4 curve $[C,q_1,q_2,p]$ such that the image of $C$ under $|K_C+q_1+q_2-p|$ does not lie on any quadrics, which we already know, since this is the $t=0$ case of our problem and in that particular case $C$ lies on a quadric only if it is Brill-Noether special as we indicated earlier in the introduction.
\end{proof}

\begin{rem}
Note that as opposed to other degeneration arguments, we did not address the problem of smoothability of the degenerate curves in the proofs of Theorem \ref{Nondegenerateness over Delta 0:k} and Theorem \ref{Nondegenerateness over Delta irr}, as the degenerate curves we consider here are already elements of the moduli space in consideration.
\end{rem}
\begin{rem}
With the same methods used in the proofs above, one can prove a finer version of the maximal rank conjecture in the case $g-r+d=1$. Namely, one can prove that for a general curve $C$ of genus $g$ and general points $p_1,\dots ,p_k$ the map
\begin{equation*}
\sym^2 H^0(K_C-\sum_{j=1}^k a_jp_j)\to H^0(K_C^{\otimes 2} - \sum_{j=1}^k 2a_jp_j)
\end{equation*}
is of maximal rank for any choice of natural numbers $a_1,\dots ,a_k$. We did not modify the proof to cover also these cases only because it would complicate the numerology in the proof further and we will not need this fact in what follows.
\end{rem}
We have proven that the vector bundle map (\ref{vbmap2}) is generically non-degenerate over $\Delta_{0:S}$, but this is no longer true over $\Delta_{1:S}$. In order to compute the coefficients $b_{1:s}(t)$ precisely (rather than just giving a lower bound for it), one needs a finer analysis of the limit points of $\overline{\mathfrak{Quad}}_{g(t),n(t)}$ inside the boundary of $\overline{\mathcal{M}}_{g,n}$. We will use limit linear series to carry out this analysis. The limiting behaviour of very similar multiplication maps over moduli spaces has been successfully studied using limit linear series in the papers \cite{EH2}, \cite{FP} and \cite{Fa2}. Here we will adapt the ideas developed in these papers to our situation. We start with some definitions.

\begin{defn}
Given a pointed smooth curve $[C,p]$ and a line bundle $L$ on it, we define the vector space $W_k(p,L)$ of symmetric tensors of $L$ with vanishing \mbox{order $\geq k$} at $p$ as follows:
We let $(a_0^L(p),\dots,a_r^L(p))$ be the vanishing sequence of $L$ at $p$ and $\{\sigma_0,\sigma_1,\dots ,\sigma_r\}\subseteq H^0(L)$ be a basis such that
\begin{equation*}
\ord_p(\sigma_i)=a_i^L(p).
\end{equation*}
Then we define
\begin{equation*}
W_k(p,L)\defeq \spann\{\sigma_i\sigma_j\mid a_i^L(p)+a_j^L(p)\geq k\}\subseteq \sym^2 H^0(L).
\end{equation*}
Moreover, for $\rho\in \sym^2 H^0(L)$ we define $\ord_p(\rho)=k$ if $\rho\in W_k(p,L)\setminus W_{k+1}(p,L)$.
\end{defn}

\begin{lem}
The definition of $W_k(p,L)$ is independent of the chosen basis.
\end{lem}
\begin{proof}
If we let $\{\sigma_0',\sigma_1',\dots ,\sigma_r'\}\subseteq H^0(L)$ be another basis with the property that $\ord_p(\sigma_i')=a_i^L(p)$ then clearly 
\begin{equation*}
\sigma_i'=\sum_{\ell=i}^{r}\lambda_\ell\sigma_\ell,\quad \lambda_\ell\in\mathbb{C}.
\end{equation*}
Therefore $\sigma_i'\sigma_j'$ can be written as a linear combination of symmetric tensors $\sigma_m\sigma_n$ where $m\geq i$ and $n\geq j$.
\end{proof}
Using very similar ideas as in \cite{EH1}, we (locally) construct a space of ``limit quadrics", which coincides with $\mathfrak{Quad}_{g(t),n(t)}$ in the smooth locus of $\overline{\mathcal{M}}_{g(t),n(t)}$ and has a concrete geometric description for its elements in the boundary:

\begin{thm}\label{Extension over Delta 1:k}
For $\emptyset \neq S\subseteq\{1,\dots,n(t)\}$, let $[E,q',\{p_j\mid j\in S\}]$ be a general pointed genus one curve and $[C,\{p_j\mid j\in S^c\}]$ a general pointed genus $g(t)-1$ curve. We let $q\in C$ and fix the nodal curve
\begin{equation*}
X_0\defeq [C\cup_{q\sim q'} E,p_1,\dots,p_{n(t)}]\in\overline{\mathcal{M}}_{g(t),n(t)}.
\end{equation*}
We further let
\begin{equation*}
\pi: X\to B,\quad \sigma_j:B\to X\textnormal{ for }j=1,\dots, n(t)
\end{equation*}
be the versal deformation space of $[X_0,p_1,\dots,p_{n(t)}]$ with $\pi^{-1}(0)=X_0$ and \mbox{$\sigma_j(0)=p_j$}. Then there exists a scheme $\mathcal{Q}\subseteq B$, which is singled out by the following geometric conditions:

If $b\in \mathcal{Q}$ and $X_b$ is smooth then the multiplication map
\begin{equation}\label{sym2 map at smooth fibers}
\sym^2 H^0(K_{X_b}-\sum_{j=1}^{n(t)}\sigma_j(b))\to H^0(K_{X_b}^{\otimes 2}-2\sum_{j=1}^{n(t)}\sigma_j(b))
\end{equation}
is not an isomorphism. If $b\in \mathcal{Q}$ and $X_b$ is a singular curve obtained by gluing $[C',q,\{\sigma_j(b)\mid j\in S^c \}]$ and $[E',q',\{\sigma_j(b)\mid j\in S\}]$ at the marked points $q$ and $q'$ then the map
\begin{equation}\label{W3 map}
W_3(q,K_{C'}-(s-1)q-\sum_{j\in S^c}\sigma_j(b))\to H^0(K_{C'}^{\otimes 2}-(2s+1)q-2\sum_{j\in S^c}\sigma_j(b))
\end{equation}
is not an isomorphism.

Moreover, every irreducible component of $\mathcal{Q}$ has dimension $\geq \dim B-1$.
\end{thm}
\begin{proof}
We let $\Delta\subseteq B$ be the locus where the node $q$ of $X_0$ is not smoothed and let $\mathscr{C}_q$ and $\mathscr{E}_q$ be the components of $\pi^{-1}(\Delta)$ containing $C\setminus q$ and $E\setminus q'$, respectively. By shrinking $B$, if necessary, we can assume that $\mathcal{O}_X(\mathscr{C}_q+\mathscr{E}_q)\cong\mathcal{O}_X$. We let
\begin{equation*}
L_C\defeq \omega_{\pi}(-s\cdot \mathscr{E}_q-\sum_{j=1}^{n(t)}\sigma_j(B))
\end{equation*}
and
\begin{equation}\label{fixed isomorphism}
L_E\defeq L_C(-(t+3)\cdot\mathscr{E}_q),
\end{equation}
where as before $t+3=g(t)-n(t)-1$. Note that the twists for the bundles are chosen in such a way that over $X_0$ the space of sections can be identified as follows:
\begin{equation*}
H^0(L_C|_{X_0})=H^0\left(K_C-(s-1)q-\sum_{j\in S^c}p_j\right)
\end{equation*}
and
\begin{equation*}
H^0(L_E|_{X_0})=H^0\left(\mathcal{O}_E((s+t+4)q'-\sum_{j\in S}p_j)\right).
\end{equation*}
We let $F_C\to B$ and $F_E\to B$ be the bundle of projective frames of the vector bundles $\pi_*L_C$ and $\pi_*L_E$ and we consider
\begin{equation*}
F\defeq F_C\times_B F_E.
\end{equation*}
The space $F$ parametrizes the data  $[b,\{\sigma_j^C\}_{j=0}^{t+3},\{\sigma_j^E\}_{j=0}^{t+3}]$, where $\{\sigma_j^C\}_{j=0}^{t+3}$ and $\{\sigma_j^E\}_{j=0}^{t+3}$ are ordered bases of the fibers of $\pi_*{L_C}$ and $\pi_*{L_E}$ at $b\in B$ up to scalars. We fix sections $\tau_C\in\mathcal{O}_X(\mathscr{C}_q)$ and $\tau_E\in \mathcal{O}_X(\mathscr{E}_q)$ that only vanish on $\mathscr{C}_q$ and $\mathscr{E}_q$, respectively. We denote by $\tilde{\sigma}_j^C$ and $\tilde{\sigma}_j^E$ the tautological bundles on $F$, whose fibers over each point are the 1-dimensional vector spaces corresponding to the frame with the same symbol. We define a subscheme $F'\subseteq F$ subject to the conditions
\begin{equation}\label{EH equations}
\tilde{\sigma}_j^C\cdot\tau_C^{t+3-j}=\tilde{\sigma}_j^E\cdot \tau_E^j
\end{equation}
as sections of the bundles
\begin{equation*}
\pi_*L_C((t+3-j)\cdot\mathscr{C}_q)\cong\pi_*L_E(j\cdot\mathscr{E}_q),
\end{equation*}
where the isomorphism is induced by (\ref{fixed isomorphism}) and the isomorphism $\mathcal{O}_X(\mathscr{C}_q+\mathscr{E}_q)\cong\mathcal{O}_X$. The resulting space $F'$ parametrizes the data $[b,\{\sigma_j^C\}_{j=0}^{t+3},\{\sigma_j^E\}_{j=0}^{t+3}]$, where $\sigma_j^C$ and $\sigma_j^E$ are identified if $b\in B\setminus \Delta$ and if $b\in \Delta$ then
\begin{equation*}
ord_q(\sigma_j^C)\geq j\textnormal{ and }ord_{q'}(\sigma_j^E)\geq t+3-j.
\end{equation*}
Note by (\ref{EH equations}) that the section $\tilde{\sigma}_j^C$ vanishes at least $j$ times along $\mathscr{E}_q$. Thus, we have an injective map
\begin{equation*}
\tilde{\sigma}_j^C\hookrightarrow\pi_*{L_C}(-j\cdot\mathscr{E}_q).
\end{equation*}
Similarly, we have that
\begin{equation*}
\tilde{\sigma}_j^E\hookrightarrow\pi_*{L_E}(-(t+3-j)\cdot\mathscr{C}_q).
\end{equation*}
Therefore for $j+k\geq 3$, we have maps
\begin{equation*}
\tilde{\sigma}_j^C\otimes \tilde{\sigma}_k^C\hookrightarrow \pi_*{L_C}(-j\cdot\mathscr{E}_q)\otimes \pi_*{L_C}(-k\cdot\mathscr{E}_q)\to \pi_*{L_C^{\otimes 2}}(-(j+k)\cdot\mathscr{E}_q)\to \pi_*{L_C^{\otimes 2}}(-3\cdot\mathscr{E}_q),
\end{equation*}
where the middle map is the usual multiplication map and the last map is induced by multiplying sections with $\tau_E^{j+k-3}$. On the other hand, for $j+k< 3$, we have maps
\begin{align*}
\tilde{\sigma}_{j}^E\otimes \tilde{\sigma}_k^E &\hookrightarrow \pi_*{L_E}(-(t+3-j)\cdot\mathscr{C}_q)\otimes \pi_*{L_E}(-(t+3-k)\cdot\mathscr{C}_q)\\
&\to \pi_*{L_E^{\otimes 2}}(-(2t+6-j-k)\cdot\mathscr{C}_q)\to\pi_*{L_E^{\otimes 2}}(-(2t+3)\cdot\mathscr{C}_q)\cong \pi_*{L_C^{\otimes 2}}(-3\cdot\mathscr{E}_q),
\end{align*}
where similarly the last map is multiplying sections with $\tau_C^{3-j-k}$ and the isomorphism is induced by (\ref{fixed isomorphism}). Next, we define the vector bundle
\begin{equation*}
S\defeq\left(\bigoplus_{\substack{j+k\geq 3 \\ k\geq j}} \tilde{\sigma}_j^C\otimes \tilde{\sigma}_k^C\right)\oplus \left(\bigoplus_{\substack{j+k<3 \\ k\geq j}}\tilde{\sigma}_{j}^E\otimes \tilde{\sigma}_k^E\right), 
\end{equation*}
and consider the vector bundle map
\begin{equation*}
\phi:S\to \pi_*{L_C^{\otimes 2}}(-3\cdot\mathscr{E}_q),
\end{equation*}
which at fibers is the map that takes the quadratic polynomials given by the individual direct summands of $S$ and evaluates their sum under the multiplication map. Note that due to the identifications (\ref{EH equations}) the fiber of this vector bundle map at a point $b\in B\setminus\Delta$ is the map in (\ref{sym2 map at smooth fibers}). Next, we describe the fiber over $b=0$. First note that the fiber of $\pi_*{L_C^{\otimes 2}}(-3\cdot\mathscr{E}_q)$ over $0$ is identified by the vector subspace of sections in
\begin{equation*}
H^0\left(K_C^{\otimes 2}-(2s+1)q-2\sum_{j\in S^c}\sigma_j(b)\right)\oplus H^0\left(\mathcal{O}_E((2s+5)q'-2\sum_{j\in S}\sigma_j(b))\right)
\end{equation*}
that are compatible at the node $q$.

The direct summands $\tilde{\sigma}_{j}^E\otimes \tilde{\sigma}_k^E$ in $S$ are multiplied by non-trivial powers of $\tau_C$ (since $3-j-k> 0$) as described above. Therefore, the sections of $\pi_*{L_C^{\otimes 2}}(-3\cdot\mathscr{E}_q)$ that are in the image of the map
\begin{equation*}
\tilde{\sigma}_{j}^E\otimes \tilde{\sigma}_k^E\to \pi_*{L_C^{\otimes 2}}(-3\cdot\mathscr{E}_q)
\end{equation*}
restrict to zero on $C$ and on $E$ they restrict to sections of 
\begin{equation*}
H^0\left(\mathcal{O}_E((2s+5)q'-2\sum_{j\in S}\sigma_j(b))\right),
\end{equation*}
that vanish at $q'$. Arguing in the same way, we observe that for $j+k>3$, sections that are in the image of
\begin{equation*}
\tilde{\sigma}_{j}^C\otimes \tilde{\sigma}_k^C\to \pi_*{L_C^{\otimes 2}}(-3\cdot\mathscr{E}_q)
\end{equation*}
restrict to zero on $E$ and on $C$ they restrict to sections of 
\begin{equation*}
H^0\left(K_C^{\otimes 2}-(2s+1)q-2\sum_{j\in S^c}\sigma_j(b)\right),
\end{equation*}
that vanish at $q$. The images of the remaining direct summands, $\tilde{\sigma}_{0}^C\otimes \tilde{\sigma}_3^C$ and $\tilde{\sigma}_{1}^C\otimes \tilde{\sigma}_2^C$ in $\pi_*{L_C^{\otimes 2}}(-3\cdot\mathscr{E}_q)$ restrict to sections on $E$ and $C$ that are compatible at the node $q\sim q'$.

It is elementary to observe that the fiber of $\phi$ at $0\in B$ always surjects onto the sections on $E$. Therefore, $\phi$ fails to be an isomorphism over $0\in B$ if and only if the map (\ref{W3 map}) is not an isomorphism. 

We define $\tilde{\mathcal{Q}}\subseteq F'$ as the locus where the map $\phi$ fails to be an isomorphism and let $\mathcal{Q}$ be the image of $\tilde{\mathcal{Q}}$ under the morphism $F\to B$. 

To estimate the dimension of $\mathcal{Q}$, first observe that the fibers of $F$ are isomorphic to two copies of the projective linear group of a vector space of dimension $t+4$. Therefore,
\begin{equation*}
\dim F=\dim B+2(t+3)(t+4).
\end{equation*}
Each of the conditions in (\ref{EH equations}) is a single equation on the elements of a projective bundle with fibers isomorphic to $\mathbb{P}^{t+3}$. Therefore, each of them imposes $t+3$ conditions. The determinantal condition on $\phi$ clearly imposes (at most) one condition. Thus, we have the estimate that every irreducible component of $\tilde{\mathcal{Q}}$ has dimension at least
\begin{equation*}
\dim F-(t+3)(t+4)-1=\dim B+(t+3)(t+4)-1.
\end{equation*}
To finish the proof, we need to show that the fiber dimension of $\tilde{\mathcal{Q}}\to B$ is at most $(t+3)(t+4)$. This is clear over $b\in B\setminus \Delta$, since in this case the frames are identified and the fiber of $F'\to B$ is isomorphic to a single copy of $\mathbb{P}GL_{t+4}$. Over $b\in \Delta$, we have the same estimate on the fiber dimension, because by the generality of the pointed elliptic curve $[E,q',\{p_j\mid j\in S\}]$, we have that 
\begin{equation*}
H^0\left(\mathcal{O}_E(s\cdot q'-\sum_{j\in S}p_j)=0\right),
\end{equation*}
which forces $ord_{q'}(\sigma_j^E)= t+3-j$ for all $j$. Similarly, by the generality of the pointed curve $[C,\{p_j\mid j\in S^c\}]$, we have that $ord_q(\sigma_j^C)=j$ for all $j$ (We are disregarding the case where $s=n(t)$ and $q$ is a Weierstrass point of $C$, because it plays no role in the dimension count). An elementary dimension count now shows that the possible frames $\{\sigma_j^C\}_{j=0}^{t+3},\{\sigma_j^E\}_{j=0}^{t+3}$ subject to conditions 
\begin{equation*}
ord_q(\sigma_j^C)+ord_{q'}(\sigma_{t+3-j}^E)=t+3,
\end{equation*}
depend on $(t+3)(t+4)$ parameters.
\end{proof}

Note that by Theorem \ref{Extension over Delta 1:k}, we have the necessary condition that if the pointed nodal curve $[X_0,p_1,\dots,p_{n(t)}]\in \overline{\mathfrak{Quad}}_{g(t),n(t)}$ then the map (\ref{W3 map}) fails to be an isomorphism. To show that this is also sufficient, one has to rule out the possibility that an irreducible component of $\mathcal{Q}$ lies in the boundary. Since we have already shown that the dimension of every irreducible component of $\mathcal{Q}$ is at least $\dim\overline{\mathcal{M}}_{g(t),n(t)}-1$, we can exclude this possibility by checking that the map (\ref{W3 map}) is generically non-degenerate over the boundary divisors $\Delta_{1:S}$. This is the content of the next theorem.
\begin{thm}\label{Nondegenerateness over Delta 1:k}
For $\emptyset\neq S\subseteq \{1,\dots , n(t)\}$ and a general pointed genus $g(t)-1$ curve $[C,q,\{p_j\mid j\in S^c\}]\in\mathcal{M}_{g(t)-1,S^c\cup\{q\}}$ the map
\begin{equation*}
W_3(q,K_C-\sum_{j\in S^c}p_j-(s-1)q)\to H^0(K_C^{\otimes 2} -2\sum_{j\in S^c}p_j-(2s+1)q)
\end{equation*}
is an isomorphism.
\end{thm}
\begin{proof}
Clearly it is sufficient to specialize to the case where the points $p_j=q$ for all $j\in S^c$ and prove that the map
\begin{equation*}
W_3(q,K_C-(n(t)-1)q)\to H^0(K_C^{\otimes 2}- 2(n(t)-1)q)
\end{equation*}
is an isomorphism for a general element $[C,q]\in\mathcal{M}_{g(t)-1,1}$.

To prove this statement, we follow the same steps of successive degenerations as in the proof of Theorem \ref{Nondegenerateness over Delta 0:k}, which (skipping the details) reduces the question to prove that there exists a pointed curve $[C,q,q_1,q_2,q_3,q_4]\in \mathcal{M}_{2,5}$ such that
\begin{equation*}
W_3(q,K_C+q_1+q_2+q_3+q_4-q)\to H^0(K_C^{\otimes 2}+2q_1+2q_2+2q_3+2q_4-2q)
\end{equation*}
is an isomorphism. 

To see this, note that if we choose the points $q,q_1,q_2,q_3,q_4\in C$ to be general, then the image of $C$ under $|K_C+q_1+q_2+q_3+q_4-q|$ is contained in a unique (rank 4) quadric, which correspond to the pencils $|K_C|$ and $|q_1+q_2+q_3+q_4-q|$, both of which have the vanishing type $(0,1)$ at the point $q$. That is, the tangent space of the quadric has multiplicity 2 at $q$ and therefore the quadric is not an element of $W_3(q,K_C+q_1+q_2+q_3+q_4-q)$.
\end{proof}

\begin{cor}
We have that $b_{1:0}(t)=t+4$ and $b_{1:1}(t)=4$.
\end{cor}
\begin{proof}
We consider the gluing map
\begin{equation*}
\nu:\overline{\mathcal{M}}_{1,2}\to \overline{\mathcal{M}}_{g(t),n(t)}
\end{equation*}
that attaches a general genus $g(t)-1$ curve $[C,q,p_1,\dots,p_{n(t)-1}]$ to $[E,p,q']\in\overline{\mathcal{M}}_{1,2}$ by identifying the points $q$ and $q'$. By Theorem \ref{Nondegenerateness over Delta 1:k}, we have that 
\begin{equation*}
\nu^*\left(\overline{\mathfrak{Quad}}_{g(t),n(t)}\right)=0.
\end{equation*}
Thus we get the relation
\begin{equation}\label{b_(1:0),b_(1:1) relation}
(8-t)\cdot\lambda -\delta_{irr}+t\cdot\psi_p+b_{1:1}(t)\cdot\psi_{q'}-b_{1:0}(t)\cdot\delta_{0:\{p,q'\}}=0
\end{equation}
in $\pic(\overline{\mathcal{M}}_{1,2})$. Among the classes $\lambda,\psi_p,\psi_{q'},\delta_{irr},\delta_{0:\{p,q'\}},$ we have the following relations (see \cite{AC}): 
\begin{equation*}
12\lambda=\delta_{irr}\quad \textnormal{and}\quad \psi_p=\psi_{q'}=\lambda+\delta_{0:\{p,q'\}}.
\end{equation*}
Using these, we can rewrite the relation (\ref{b_(1:0),b_(1:1) relation}) as
\begin{equation*}
(b_{1:1}(t)-4)\cdot \lambda+(b_{1:1}(t)-b_{1:0}(t)+t)\cdot \delta_{0:\{p,q'\}}=0,
\end{equation*}
from which the statement clearly follows.
\end{proof}
\begin{cor}
We have that $b_{1:s}(t)= \frac{1}{2} (s^2t+s^2 -st + s + 6)$ for $s\geq 1$.
\end{cor}
\begin{proof}
Intersecting the test curve $T_{1:S}$ with $\overline{\mathfrak{Quad}}_{g(t),n(t)}$ we obtain the relation
\begin{equation}\label{C1k relations}
T_{1:S}\cdot\overline{\mathfrak{Quad}}_{g(t),n(t)}=t(n(t)-s)+(2g(t)-4+n(t)-s)b_{1:s}(t)-(n(t)-s)b_{1:s+1}(t).
\end{equation}
The construction of the space $\mathcal{Q}$ in Theorem (\ref{Extension over Delta 1:k}) can be carried out with obvious modifications in the special case, where $q\in C$ and one of the marked points $p_j\in C$ come together (See the computation in the next paragraph for the counterpart of the condition (\ref{W3 map}) in this case). This enables us to compute the left hand side of the equation (\ref{C1k relations}):

We consider the maps
\begin{center}

\begin{tikzcd}[column sep={{{{4em,between origins}}}}]
 & C\times C\arrow{dl}[swap]{\pi_1}\arrow{dr}{\pi_2} & \\
 C & & C
\end{tikzcd}
\end{center}
and let $\Delta\defeq\{(p,p)\in C\times C\mid p\in C\}$ and 
\begin{equation*}
L\defeq {\pi_1}^*\left(\omega_C(-\sum_{j\in S^c}p_j)\right)\otimes \mathcal{O}(-(s-1)\Delta).
\end{equation*} 
The intersection number $T_{1:S}\cdot \overline{\mathfrak{Quad}}_{g(t),n(t)}$ is equal to the class of the degeneracy locus of the vector bundle map
\begin{equation*}
W_3(\sym^2 {\pi_2}_*L)\xrightarrow{\theta} {\pi_2}_*\left(L^{\otimes 2}\otimes \mathcal{O}(-3\Delta)\right).
\end{equation*}
The bundle $W_3(\sym^2 {\pi_2}_*L)$ sits naturally in the following exact sequences:
\begin{equation*}
0\to \sym^2 {\pi_2}_*\left(L(-2\Delta)\right)\to W'\to {\pi_2}_*\left(L(-2\Delta)\right)\otimes {\pi_2}_*\left( L\otimes\mathcal{I}_\Delta/\mathcal{I}^2_\Delta\right) \to 0,
\end{equation*}
and
\begin{equation*}
0\to W'\to W_3(\sym^2 {\pi_2}_*L)\to {\pi_2}_*\left(L(-3\Delta)\right)\otimes {\pi_2}_*\left( L\otimes\mathcal{O}_\Delta\right)\to 0.
\end{equation*}
Using these exact sequences and applying G-R-R and Porteous formula, we compute
\begin{equation*}
[D_1(\theta)]=\frac{1}{2} \left(s^2 \left(t^3+6 t^2+13 t+8\right)-2 s \left(t^3+4 t^2+4 t-3\right)+t^3+8 t^2+29 t+34\right).
\end{equation*}
Using this equation we solve the recurrence relation (\ref{C1k relations}) and obtain that
\begin{equation*}
b_{1:s}(t)=\frac{1}{2} (s^2 t+s^2 -st + s + 6).
\end{equation*}
\end{proof}
\section{Kodaira Dimension of $\overline{\mathcal{M}}_{g,n}$}
In this section we use Theorem \ref{main} to obtain some results about the Kodaira dimension of $\overline{\mathcal{M}}_{g,n}$. The canonical class of $\overline{\mathcal{M}}_{g,n}$ is well known to be
\begin{equation*}
K_{\overline{\mathcal{M}}_{g,n}}=13\lambda-2\delta_{irr}+ \sum_{j=1}^n\psi_j-2\sum_{\substack{S\in P\\|S|\geq 2}}\delta_{0:S}-3\sum_{S\in P}\delta_{1:S}-2\sum_{i=2}^{\left \lfloor{g/2}\right \rfloor}\sum_{S\in P}\delta_{i:S},
\end{equation*}
where $P$ denotes the power set of $\{1,\dots,n\}$. Since the class $\sum_{j=1}^n\psi_j$ is \emph{big}, expressing the canonical divisor $K_{\overline{\mathcal{M}}_{g,n}}$ as a positive linear combination of $\sum_{j=1}^n\psi_j$ and other effective divisors implies that $\overline{\mathcal{M}}_{g,n}$ is of general type.
\begin{proof}[Proof of Theorem \ref{GenTypeThms}]
We consider the map
\begin{equation*}
\nu_{i,j}:\overline{\mathcal{M}}_{16,8}\to \overline{\mathcal{M}}_{17,10}
\end{equation*}
that attaches an elliptic curve with two marked points to the point labeled by $i$ and a rational curve with two marked points to the point labeled by $j$.

We want to pullback the divisor $\overline{\mathfrak{Quad}}_{17,10}$ via this gluing map. To ensure that we obtain an effective divisor on $\overline{\mathcal{M}}_{16,8}$ this way, one needs to check that the map
\begin{equation*}
W_3(p_i,K_C-p_i-2p_j-\sum_{\ell\neq i,j}^8 p_\ell)\to H^0(K_C^{\otimes 2} -2p_i-4p_j-2\sum_{\ell\neq i,j}^8 p_\ell )
\end{equation*}
is an isomorphism for a general element $[C,p_1,\dots,p_8]\in\mathcal{M}_{16,8}$, so that the image of $\nu_{i,j}$ is not contained in $\overline{\mathfrak{Quad}}_{17,10}$. This clearly follows from Theorem \ref{Nondegenerateness over Delta 1:k}.

Using the pullback formulas in \cite{AC}, we compute
\begin{equation*}
\nu_{i,j}^*\left(\overline{\mathfrak{Quad}}_{17,10}\right)=5\lambda+3\sum_{\ell\neq i,j}\psi_\ell-\delta_{irr}+9\psi_i+10\psi_j-\dots
\end{equation*}

We compute this pullback for every choice of markings $\{i,j\}\subseteq \{1,\dots,8\}$ and take the average of the resulting divisors to obtain the effective class
\begin{equation*}
D_{16,8}=40\lambda+ 37\sum_{j=1}^8\psi_j -8\delta_{irr}-\dots
\end{equation*}
Next, we consider the pullback of the effective divisor $\mathcal{Z}_{16}\subseteq\overline{\mathcal{M}}_{16}$, which is defined as the closure of the locus of curves $[C]\in\mathcal{M}_{16}$ that are contained in a quadric under the map given by a $g^7_{21}$. The class of $\mathcal{Z}_{16}$ is computed in \cite{Fa2}:
\begin{equation*}
\mathcal{Z}_{16}=407\lambda-61\delta_{irr}-\dots
\end{equation*}
Using these two classes, the canonical class of $\mathcal{M}_{16,8}$ can be written as
\begin{equation*}
K_{\overline{\mathcal{M}}_{16,8}}=\frac{13}{272}\sum_{j=1}^8\psi_j +  \frac{7}{272}D_{16,8}+\frac{1}{34}\mathcal{Z}_{16}+E,
\end{equation*}
where $E$ an effective divisor supported on $\overline{\mathcal{M}}_{16,8}\setminus (\mathcal{M}_{16,8}\cup \Delta_{irr})$.

To obtain an analogous description for $K_{\overline{\mathcal{M}}_{17,8}}$, we consider the map
\begin{equation*}
\nu_{i,j}:\overline{\mathcal{M}}_{17,8}\to \overline{\mathcal{M}}_{17,10}
\end{equation*}
that attaches a genus zero curve with two marked points to each of the points labeled by $i$ and $j$. By Theorem \ref{Nondegenerateness over Delta 0:k}, we have that
\begin{equation*}
\nu_{i,j}(\overline{\mathcal{M}}_{17,8})\not\subseteq\overline{\mathfrak{Quad}}_{17,10}.
\end{equation*}
Therefore, the class
\begin{equation*}
\nu_{i,j}^*\left(\overline{\mathfrak{Quad}}_{17,10}\right)=5\lambda+3\sum_{\ell\neq i,j}\psi_\ell-\delta_{irr}+10\psi_i+10\psi_j-\dots
\end{equation*}
is an effective divisor on $\overline{\mathcal{M}}_{17,8}$. As before, we apply this procedure for every $\{i,j\}\subseteq \{1,\dots,8\}$ and take the average to obtain the divisor
\begin{equation*}
D_{17,8}=20\lambda+19\sum_{j=1}^8\psi_j-4\delta_{irr}-\dots
\end{equation*}
The pullback of the Brill-Noether divisor on $\overline{\mathcal{M}}_{17}$ to $\overline{\mathcal{M}}_{17,8}$ has the class
\begin{equation*}
BN_{17}=20\lambda-3\delta_{irr}-\dots
\end{equation*}
Using these, we can write
\begin{equation*}
K_{\overline{\mathcal{M}}_{17,8}}=\frac{1}{20}\sum_{j=1}^8\psi_j+\frac{1}{20}D_{17,8}+\frac{3}{5}BN_{17}+E,
\end{equation*}
where as before $E$ is an effective divisor supported on $\overline{\mathcal{M}}_{17,8}\setminus (\mathcal{M}_{17,8}\cup \Delta_{irr})$.

Finally for $\overline{\mathcal{M}}_{12,10}$, we use two divisor classes computed by Farkas and Verra. The first one is defined as the closure of the locus of curves in $\mathcal{M}_{12}$, which are contained in a quadric under the map given by a $g^4_{14}$. The class of this divisor is computed in \cite{FV1}:
\begin{equation*}
\overline{\mathfrak{D}}_{12} = 13245\lambda - 1926 \delta_{irr} - \dots
\end{equation*}
The second one is given as the closure of the locus of pointed curves $[C,p_1,\dots,p_{10}]$ in $\mathcal{M}_{12,10}$, where the curve admits a map $C\to \mathbb{P}^1$ of degree $11$, where the points $p_1,\dots p_{10}$ lie in a fiber. The class of this divisor is computed in \cite{FV2}:
\begin{equation*}
\overline{\mathcal{F}}_{12,10} = 9\sum_{j=1}^{10} \psi_j - \delta_{irr} - \dots 
\end{equation*}
Now the canonical class of $\overline{\mathcal{M}}_{12,10}$ can be written as
\begin{equation*}
K_{\overline{\mathcal{M}}_{12,10}}=\frac{59}{4415}\sum_{j=1}^{10} \psi_j + \frac{13}{13245}\overline{\mathfrak{D}}_{12} + \frac{484}{4415}\overline{\mathcal{F}}_{12,10} + E,
\end{equation*}
with $E$ being an effective divisor supported on $\overline{\mathcal{M}}_{12,10}\setminus (\mathcal{M}_{12,10}\cup \Delta_{irr})$.
\end{proof}

\bibliographystyle{amsalpha}


\end{document}